\documentclass[12pt]{elsarticle}

\usepackage{amssymb}
\usepackage{amsmath}
\usepackage[german,english]{babel}
\usepackage{graphicx}
\usepackage{subfigure}
\usepackage{cases}
\usepackage{amsfonts}
\usepackage{amsthm}
\usepackage{booktabs}
\usepackage{hyperref}
\usepackage{epstopdf}
\usepackage{color}
\usepackage{float}

\usepackage{mathabx}
\usepackage{mathrsfs}
\usepackage{dsfont}
\usepackage{extarrows}  

\newtheorem{assumption}{Assumption}

\newtheorem{remark}{Remark}
\newtheorem{lemma}{Lemma}
\newtheorem{theorem}{Theorem}

\let \mr=\mathrm
\let \b=\boldsymbol

\begin{document}

\begin{frontmatter}

\title{Analysis of the SDFEM in a modified streamline diffusion norm for singularly perturbed convection diffusion problems\tnoteref{label3}}
\author[label2]{Jin Zhang}
\author[label1]{Xiaowei Liu}

\address[label2]{School of Mathematical Sciences, Shandong Normal University,
Jinan 250014, China}
\address[label1]{College of Science, Qilu University of Technology, Jinan 250353, China}

\tnotetext[label3]{This research was partly supported by a project of Shandong province higher educational science and technology program (J16LI10)}

\begin{abstract}
In this paper we analyze a streamline diffusion finite element method (SDFEM)
for   a model singularly perturbed convection diffusion problem. To put insight into the influences of the stabilization parameter  on SDFEM solutions, we discuss how to obtain the uniform estimates of errors in the streamline diffusion norm.
 By decreasing the standard stabilization parameters properly near the exponential layers, we obtain the uniform estimates in a norm, which is stronger than the $\varepsilon-$energy norm and weaker than the standard streamline diffusion norm.  

\end{abstract}


\end{frontmatter}

%
%
%

%
%
\section{Introduction}
Consider the following singularly perturbed boundary value problem:
 \begin{equation}\label{eq:model problem}
 \begin{array}{rcl}
-\varepsilon\Delta u+\boldsymbol{b}\cdot \nabla u+cu=f  & \mbox{in}& \Omega=(0,1)^{2},\\
 u=0 & \mbox{on}& \partial\Omega,
 \end{array}
 \end{equation}
where $\varepsilon\ll |\b{b}|$ is a positive parameter, $\boldsymbol{b}(x,y)=(b_{1}(x,y),b_{2}(x,y))^{T}$,  and $c(x,y)$ and $f(x,y)$  are supposed sufficiently smooth. Also we assume that 
$$
b_{1}(x,y)\ge \beta_{1}>0,\,b_{2}(x,y)\ge \beta_{2}>0,c(x,y)\ge 0\;\text{ on $\bar{\Omega}$, }
$$
and
$$
c(x,y)-\frac{1}{2}\nabla\cdot \b{b}(x,y)\ge \mu_{0}>0\quad \text{on $\bar{\Omega}$}, 
$$
where $\beta_{1}$, $\beta_{2}$, and  $\mu_{0}$ are some constants. These assumptions ensure that  problem \eqref{eq:model problem} has a unique solution
in $H^1_0(\Omega)\cap H^2(\Omega)$ for all $f\in L^2(\Omega)$  (see, e.g., \cite{Styn1Tobi2:2003-SDFEM,Zhang:2003-Finite}). In general there exist
two exponential layers of width $O(\varepsilon\ln(1/\varepsilon))$ at the sides
$x=1$ and $y=1$ for the solution to problem \eqref{eq:model problem}.

\par
For the convection-diffusion problem, we can obtain  discrete solutions with satisfactory stability and accuracy by means of stabilized methods and a priori adapted meshes (see \cite{Stynes:2005-Steady,Roo1Sty2Tob3:2008-Robust}), 
for example, the streamline diffusion finite element method (SDFEM)  \cite{Hugh1Broo2:1979-multidimensional} and a Shishkin mesh \cite{Shishkin:1990-Grid}.  
For the SDFEM on Shishkin rectangular meshes,  convergence properties have been widely studied and the reader is referred to \cite{Styn1ORior2:1997-uniformly,Styn1Tobi2:2003-SDFEM,Fra1Lin2Roo3:2008-Superconvergence,Liu1Zhan2:2015-Analysis,Zhan1Liu2:2016-Convergence-CL} and references therein. 

It is easy to obtain  uniform bounds  of $u-u^N$ in the $\varepsilon-$energy norm defined in \eqref{eq: energy norm} (see \cite{Styn1ORior2:1997-uniformly,Styn1Tobi2:2003-SDFEM}), where $u$ is the solution to problem \eqref{eq:model problem} and $u^{N}$ is the SDFEM solution.   Compared with the $\varepsilon-$energy norm, the streamline diffusion norm $\Vert \cdot \Vert_{SD}$ defined in \eqref{eq: SD norm}  is more proper to measure the energy properties of the SDFEM solution, which is derived from the bilinear form of the SDFEM.  Nevertheless, it is impossible to obtain a uniform bound of $\Vert u-u^{I} \Vert_{SD}$ where $u^{I}$ is the interpolant of the solution $u$ from the finite element space of piecewise bilinears, as can be seen by a simple one-dimensional example. The reason lies in the estimates of the term  $\sum \Vert \delta^{1/2}\nabla (u-u^{I}) \Vert$ in $\Vert u- u^{I}\Vert_{SD}$: there is always a factor $\varepsilon^{-1}$,  which can not be balanced by the stabilization parameter $\delta=O(N^{-1})$ ($N$ is the mesh parameter).

In this paper, by decreasing the stabilization parameter near the exponential layers, we obtain uniform bounds of $u-u^N$ in a modified streamline diffusion norm 
 which is stronger than $\varepsilon-$energy norm but weaker than the standard streamline diffusion norm. With this modification of stabilization parameters, numerical stability of the SDFEM is preserved, as can be observed from  numerical tests. 
Besides, we obtain the following uniform local estimates: 
\begin{equation*}
\Vert u-u^{N} \Vert_{SD;\,\Omega_{s}}
\le CN^{-3/2},\quad \Vert u-u^{N}\Vert_{\varepsilon;\,\Omega_{s}} \le C (\varepsilon^{1/2}N^{-1}+N^{-2}\ln^{2}N),
\end{equation*}
where $\Omega_s$ can be seen in Figure \ref{fig:Shishkin mesh I}.

 Here is the outline of this article. In \S 2 we give some a priori information for the solution of \eqref{eq:model problem}, then introduce the Shishkin meshes, a streamline diffusion finite element method on these meshes and our new stabilization parameters. In \S 3 we obtain the global and local estimates.  Finally, some numerical results are presented in \S 4.

Throughout this paper, $C$ will denote a generic positive constant, not necessarily the same at each occurrence,
which is independent of $\varepsilon$
and of the mesh parameter $N$.

%
%
\section{The SDFEM on Shishkin meshes}
For the convenience of reading we will present some basic knowledge in this section including the Shishkin meshes, the  SDFEM and some assumptions.

\subsection{Shishkin meshes}\label{subsection-Shishkin meshes}
We use the \textit{Shishkin} meshes to discretize \eqref{eq:model problem}, that is, there are both $N$ (a positive even integer) mesh intervals in $x-$ and $y-$direction which amass in the layer regions.
For this purpose we assume that $\varepsilon\le N^{-1}$ and define the parameters
\begin{equation*}
\lambda_{x}:=\rho\frac{\varepsilon}{\beta_{1}}\ln N,   \quad
\lambda_{y}:=
\rho\frac{\varepsilon}{\beta_{2}}\ln N
\end{equation*}
where $\rho=2.5$.

\begin{figure}
\centering
\includegraphics[width=3.5in]{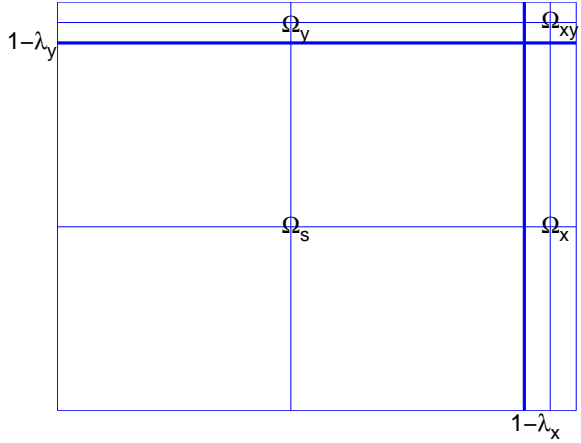}
\caption{Shishkin meshes}
\label{fig:Shishkin mesh I}
\end{figure}

The domain $\Omega$ is separated into four parts as $\bar{\Omega}=\Omega_{s}\cup\Omega_{x}\cup\Omega_{y}\cup\Omega_{xy}$ (see Figure \ref{fig:Shishkin mesh I}), where
\begin{align*}
&\Omega_{s}:=\left[0,1-\lambda_{x}\right]\times\left[0,1-\lambda_{y}\right],&&
\Omega_{x}:=\left[ 1-\lambda_{x},1 \right]\times\left[0,1-\lambda_{y}\right],\\
&\Omega_{y}:=\left[0,1-\lambda_{x}\right]\times\left[1-\lambda_{y},1 \right],&&
\Omega_{xy}:=\left[ 1-\lambda_{x},1 \right]\times\left[1-\lambda_{y},1 \right].
\end{align*}
Define
\begin{numcases}{x_{i}=}
2i(1-\lambda_{x})/N ,&\text{for $i=0,\,\cdots,\,N/2$}, \nonumber\\
1-2(N-i)\lambda_{x}/N, &\text{for $i=N/2+1,\,\cdots,\,N$},\nonumber
\end{numcases}
and
\begin{numcases}{y_{j}=}
2j(1-\lambda_{y})/N ,&\text{for $j=0,\,\cdots,\,N/2$}, \nonumber\\
1-2(N-j)\lambda_{y}/N, &\text{for $j=N/2+1,\,\cdots,\,N$}.\nonumber
\end{numcases}

Draw lines through these mesh points parallel to the $x$-axis and $y$-axis, and the domain $\Omega$ is dissected into rectangles. This yields a triangulation of $\Omega$  denoted by $\mathcal{T}_{N}$ (see  Figure \ref{fig:Shishkin mesh I}).
If $K=K_{ij}:=[x_{i},x_{i+1}]\times[y_{j},y_{j+1}]$, the mesh sizes $h_{x,K}:=x_{i+1}-x_{i}$ and $h_{y,K}:=y_{j+1}-y_{j}$ satisfy
\begin{numcases}{h_{x,K}=}
H_{x}:=\frac{1-\lambda_{x}}{N/2}, &\text{for $i=0,\,\cdots,\,N/2-1$}, \nonumber\\
h_{x}:=\frac{\lambda_{x}}{N/2}, &\text{for $i=N/2,\,\cdots,\,N-1$}, \nonumber
\end{numcases}
and
\begin{numcases}{h_{y,K}=}
H_{y}:=\frac{1-\lambda_{y}}{N/2}, &\text{for $j=0,\,\cdots,\,N/2-1$},\nonumber \\
h_{y}:=\frac{\lambda_{y}}{N/2}, &\text{for $j=N/2,\,\cdots,\,N-1$}.\nonumber
\end{numcases}


\subsection{The streamline diffusion finite element method}
For any subdomain $D$ of $\Omega$, denote the standard (semi-)norms in $H^{1}(D)$ and  $L^{2}(D)$ by $\vert \cdot \vert_{1,D}$ and  $\Vert\cdot\Vert_{D}$ respectively. If $D=\Omega$ then we remove $\Omega$ from the notation.

Let $V:=H^{1}_{0}(\Omega)$. A variational formulation of problem \eqref{eq:model problem} reads as
\begin{equation}\label{weak formulation}
\left\{
\begin{array}{lr}
\text{Find $u\in V$ such that for all $v\in V$}\\
\varepsilon (\nabla u,\nabla v)+(\boldsymbol{b}\cdot\nabla u+cu,v)=(f,v).
\end{array}
\right.
\end{equation}
Obviously there is a unique solution of the weak formulation \eqref{weak formulation} by Lax-Milgram Lemma.
\par
On the Shishkin meshes in the above subsection, we define a $C^0$ bilinear finite element space as follows:
\begin{equation*}
V^{N}:=\{v^{N}\in C(\bar{\Omega}):v^{N}|_{\partial\Omega}=0
 \text{ and $v^{N}|_{K}\in Q_{1}(K)$, }
  \forall K\in \mathcal{T}_{N} \}.
\end{equation*}

\par
The standard Galerkin discretisation of \eqref{weak formulation} reads:
\begin{equation*}
\left\{
\begin{array}{lr}
\text{Find $u^{N}\in V^{N}$ such that for all $v^{N}\in V^{N}$}\\
a_{Gal}(u^{N},v^{N})=(f,v^{N}),
\end{array}
\right.
\end{equation*}
where
$$a_{Gal}(u^{N},v^{N})=\varepsilon (\nabla u^{N},\nabla v^{N})+(\boldsymbol{b}\cdot\nabla u^{N}+cu^{N},v^{N}).$$
An energy norm associated with $a_{Gal}(\cdot,\cdot)$ reads:
\begin{equation}\label{eq: energy norm}
\Vert v^{N} \Vert^{2}_{\varepsilon}:=
\varepsilon \vert v^{N} \vert^{2}_{1}+ \mu_{0}\Vert v^{N} \Vert^{2}.
\end{equation}

The SDFEM adds a stabilization term in a consistent way to the standard Galerkin discretisation and it reads:
\begin{equation*}
\left\{
\begin{array}{lr}\label{SDFEM}
\text{Find $u^{N}\in V^{N}$ such that for all $v^{N}\in V^{N}$},\\
 a_{SD}(u^{N},v^{N})=(f,v^{N})+\underset{K\subset\Omega}\sum(f,\delta\boldsymbol{b}\cdot\nabla v^{N})_{K},
\end{array}
\right.
\end{equation*}
where
\begin{equation*}
a_{SD}(u^{N},v^{N})=a_{Gal}(u^{N},v^{N})+a_{stab}(u^{N},v^{N})
\end{equation*}
and
\begin{equation*}
a_{stab}(u^{N},v^{N})=\sum_{K\subset\Omega}(-\varepsilon\Delta u^{N}+\boldsymbol{b}\cdot\nabla u^{N}+cu^{N},\delta\boldsymbol{b}\cdot\nabla v^{N})_{K}.
\end{equation*}
Note that $\Delta u^N=0$ in $K$ for $u^N\vert_K\in Q_1(K)$. Here $\delta=\delta(x,y)$ is a
user-chosen parameter which will be defined later.
We define the streamline diffusion norm (SD norm)  associated with $a_{SD}(\cdot,\cdot)$:
\begin{equation}\label{eq: SD norm}
\Vert v^{N} \Vert^{2}_{SD}:=
\varepsilon \vert v^{N} \vert^{2}_{1}+\mu_{0}\Vert v^{N} \Vert^{2}
+\sum_{K\subset\Omega} \Vert \delta^{1/2}
\b{b}\cdot\nabla v^{N}\Vert^{2}_{K}.
\end{equation}

Set
$$
x_t:=1-\lambda_x, \ x_s:=x_t-H_x, \
y_t:=1-\lambda_y, \ y_s:=y_t-H_y
$$
and
$$ 
\Omega_{s,\varepsilon}=(0,x_s)\times (0,y_s), \ \Omega^{c}_{s,\varepsilon}=\Omega_{s}\setminus \Omega_{s,\varepsilon}.$$
Note that $\mr{meas}\,(\Omega^{c}_{s,\varepsilon})\le CN^{-1}$. 
For uniform estimates in the SD norm, we set
$$
\xi(x)=
\left\{
\begin{matrix}
1&x\in [0,x_s]\\
\frac{ x_t-x } { H_x }&x\in [x_s,x_t],
\end{matrix}
\right.
\quad
\eta(y)=
\left\{
\begin{matrix}
1&y\in [0,y_s]\\
\frac{ y_t-y } { H_y }&y\in [y_s,y_t]
\end{matrix}
\right.
$$
and 
\begin{equation}\label{eq: delta-K-SD-I}
\delta(x,y):=
\left\{
\begin{split}
&C^{\ast}N^{-1}\xi(x)\eta(y)\quad&&\text{if $(x,y)\in\Omega_s$},\\
&0\quad&&\text{if $(x,y)\in\Omega\setminus\Omega_s$},
\end{split}
\right.
\end{equation}
where $C^{\ast}$ is a properly defined positive constant (see \cite[Lemma 3.25]{Roo1Sty2Tob3:2008-Robust}).

\begin{remark}
The  definition \eqref{eq: delta-K-SD-I}  is different from the usual one (see  \cite{Roo1Sty2Tob3:2008-Robust})
\begin{equation}\label{eq: delta-K-usual-I}
\delta(x,y):=
\left\{
\begin{array}{ll}
C^{\ast}N^{-1},&\text{if $(x,y)\in\Omega_{s}$},\\
0,&\text{otherwise}.
\end{array}
\right.
\end{equation}
In fact they are different from each other only in $\Omega^c_{s,\varepsilon}$. The uniform interpolation estimates in Lemma \ref{lem:u-uI in SD norm} depend on the  definition in $\Omega^c_{s,\varepsilon}$ of \eqref{eq: delta-K-SD-I} (see  \eqref{eq:II4} and \eqref{eq:integral-I}--\eqref{eq:integral-III}). 

Note the SD norm \eqref{eq: SD norm} is stronger than the $\varepsilon$-energy norm \eqref{eq: energy norm}. Clearly,  the SD norm with \eqref{eq: delta-K-SD-I} is a little weaker than one with \eqref{eq: delta-K-usual-I}.
\end{remark}

For any subdomain $D$ of $\Omega$,  notations $\Vert\cdot\Vert_{\varepsilon;\,D}$ and $\Vert \cdot \Vert_{SD;\, D}$  mean that the
integrations in \eqref{eq: energy norm} and \eqref{eq: SD norm} are restricted in $D$.


\subsection{The regularity results and interpolation bounds}
In this paper we always assume that the solution of \eqref{eq:model problem} consists of a regular solution component and various layer parts as follows.
\begin{assumption}\label{assumption-regularity} 
For our analysis we shall assume that the solution of \eqref{eq:model problem} can be decomposed as
\begin{equation}\label{eq:(2.1a)}
u=S+E_{1}+E_{2}+E_{12}.
\end{equation}
For all $(x,y)\in\bar{\Omega}$, the regular part $S$ and the layer terms $E_1$, $E_2$ and $E_{12}$ satisfy
\begin{equation}\label{eq: bounds-I}
\begin{split}
&\left| \partial^{i+j}_{x,y}S  \right|\le C \quad \text{for $0\le i+j \le 3$},\\
&\left| \partial^{i+j}_{x,y}  E_1 \right|\le C\varepsilon^{-i}e^{-\beta_{1}(1-x)/\varepsilon} \quad \text{for $0\le i,j \le 2$},\\
&\left|\partial^{i+j}_{x,y}  E_2 \right|\le C\varepsilon^{-j}e^{-\beta_{2}(1-y)/\varepsilon} \quad \text{for $0\le i,j \le 2$},\\
&\left|\partial^{i+j}_{x,y}  E_{12}  \right|\le C\varepsilon^{-(i+j)}e^{-(\beta_{1}(1-x)+\beta_{2}(1-y))/\varepsilon}\quad \text{for $0\le i,j \le 2$},
\end{split}
\end{equation}
where  $\partial^{i+j}_{x,y}v:=\dfrac{\partial^{i+j}v}{\partial x^{i}\partial y^{j}}$.
Furthermore,  assume that $S\in H^3(\Omega)$ with 
$$\Vert S \Vert_{H^3(\Omega)} \le C.$$
\end{assumption}
\begin{remark}
The conditions that ensure the above assumption valid can be found in \cite[Theorem 5.1]{Linb1Styn2:2001-Asymptotic} and \cite{Kell1Styn2:2005-Corner,Kell1Styn2:2007-Sharpened}.
\end{remark}

The following bounds will be frequently used.
\begin{lemma}\label{lem: bilinear interpolation}
Let $u^{I}$ and $E^{I}$ denote the piecewise bilinear
interpolation of $u$ and $E$ on the Shishkin mesh $\mathcal{T}_{N}$ respectively, where $E=E_{1}+E_{2}+E_{12}$. Suppose that $u$ satisfies Assumption \ref{assumption-regularity}. Then we have
\begin{align*}
&\Vert u-u^{I} \Vert_{L^{\infty}(K)}\le
\left\{
\begin{array}{ll}
C\max\{N^{-2},N^{-\rho}\} &\text{if $K \subset\Omega_{s}$},\\
CN^{-2}\ln^{2}N  &\text{otherwise},
\end{array}
\right.\\
&\Vert \nabla (u-u^{I}) \Vert_{L^{1}(\Omega_{s})}\le CN^{-1}, \quad
\Vert \nabla E^{I} \Vert_{L^{1}(\Omega_{s})}\le CN^{-\rho}.
\end{align*}
\end{lemma}
\begin{proof}
See the details in \cite[Theorem 4.2]{Styn1ORior2:1997-uniformly} for the first estimate.
The reader is referred to \cite[Lemma 3.2]{Styn1Tobi2:2003-SDFEM} for the remained bounds.
\end{proof}

%
%
\section{Interpolation and error estimates in the SD norm}
\begin{lemma}\label{lem:u-uI in SD norm}
Let Assumption \ref{assumption-regularity} hold true and $\delta$ be defined in \eqref{eq: delta-K-SD-I}, we have
 $$\Vert u-u^{I}\Vert_{SD} \le CN^{-1}\ln N.$$

\end{lemma}
\begin{proof}
From \eqref{eq: energy norm} and \eqref{eq: SD norm}, we obtain
\begin{align*}
\Vert u-u^{I} \Vert^{2}_{SD} &=\Vert u-u^{I} \Vert^{2}_{\varepsilon}+
\sum_{K\subset\Omega} (\b{b} \cdot\nabla (u-u^{I} ),\delta \b{b} \cdot\nabla (u-u^{I} ))_{K}\\
&=:\mr{I}+\mr{II}.
\end{align*}
The bound of  $\mr{I}$ can be found in \cite[Theorem 4.3]{Styn1ORior2:1997-uniformly}, that is,
\begin{equation}\label{eq: u-uI varepsilon}
|\mr{I}|\le CN^{-2}\ln^{2}N.
\end{equation}
Using the decomposition \eqref{eq:(2.1a)},  we have
\begin{align*}
u-u^{I} = S-S^{ I }+E-E^{ I },
\end{align*}
where $E=E_{1}+E_{2}+E_{12}$.
Note that $\delta=0$ for $(x,y)\in\Omega\setminus\Omega_{s}$.
Then, we rewrite $\mr{II}$ as follows:
\begin{align*}
\mr{II}=&\sum_{K\subset\Omega_s}(\b{b} \cdot\nabla (S-S^{ I } ),\delta\b{b} \cdot\nabla (u-u^{I} ))_{K}+\sum_{K\subset\Omega_s} (-\b{b} \cdot\nabla E^{ I },\delta \b{b} \cdot\nabla (u-u^{I} ))_{K}\\
&+\sum_{K\subset\Omega_s}(\b{b} \cdot\nabla E,\delta \b{b} \cdot\nabla (S-S^{ I } -E^I)  )_{K}+\sum_{K\subset\Omega_s}(\b{b} \cdot\nabla E,\delta \b{b} \cdot\nabla E)_{K}\\
=:&\mr{II}_1+\mr{II}_2+\mr{II}_3+\mr{II}_4.
\end{align*}
We will estimate $\mr{II}$ term by term. 

Note that $\delta\le CN^{-1}$ for $(x,y)\in\Omega_{s}$. By standard interpolation theories, the inequalities \eqref{eq: bounds-I} and Lemma \ref{lem: bilinear interpolation} we have
\begin{equation}\label{eq:II1}
|\mr{II}_1|\le
CN^{-1}\Vert  \nabla (S-S^{ I } ) \Vert_{L^{\infty}(\Omega_{s})} \Vert  \nabla (u-u^{I} ) \Vert_{L^{1}(\Omega_{s})}
\le C N^{-3}.
\end{equation}
Inverse estimates \cite[Theorem 3.2.6]{Ciarlet:1978-finite} and Assumption \ref{assumption-regularity} yield
\begin{equation}\label{eq:nabla EI}
\Vert \nabla E^{I}\Vert_{L^{\infty}(\Omega_s)}
\le C N \Vert E^{I}\Vert_{L^{\infty}(\Omega_s)}
\le C N^{1-\rho}.
\end{equation}
Lemma \ref{lem: bilinear interpolation} and the bound \eqref{eq:nabla EI} yield
\begin{equation}\label{eq:II2}
|\mr{II}_2|
\le
CN^{-1} \Vert \nabla E^{I}\Vert_{L^{\infty}(\Omega_s)}  \Vert \nabla(u-u^{I}) \Vert_{L^{1}(\Omega_{s})}
\le C N^{-(1+\rho)},
\end{equation}
and
\begin{equation}\label{eq:II3}
\begin{split}
|\mr{II}_3|\le &C N^{-1}\Vert  \nabla E \Vert_{L^{1}(\Omega_s)}
\left(\Vert \nabla(S-S^{ I })\Vert_{L^{\infty}(\Omega_s)}
+\Vert \nabla E^{I}\Vert_{L^{\infty}(\Omega_s)}\right)\\
\le &C (N^{-2}+N^{-\rho})\Vert  \nabla E \Vert_{L^{1}(\Omega_s)}
\le C N^{-(2+\rho)}.
\end{split}
\end{equation}
To bound the term $\mr{II}_4$, we present the following estimates first.
Note that  $0\le xe^{-x}\le e^{-1}$ for $x\ge 0$. Then  we have 
\begin{equation}\label{eq:integral-I}
\begin{split}
\int_0^{x_s}e^{-\frac{2\beta_1(1-x)}{\varepsilon}  }\mr{d}x
&=\int_0^{x_s}e^{-\frac{2\beta_1(1-x_t)}{\varepsilon}  } e^{-\frac{2\beta_1(x_t-x)}{\varepsilon}  }\mr{d}x\\
&=
N^{-2\rho} \frac{\varepsilon}{2\beta_1} 
\left( e^{-\frac{2\beta_1H_x}{\varepsilon}  } -e^{-\frac{2\beta_1x_t}{\varepsilon}  }\right)\\
&\le
N^{-2\rho}  \left(\frac{\varepsilon}{2\beta_1}\right)^2H^{-1}_x\cdot  \frac{2\beta_1H_x}{\varepsilon} e^{-\frac{2\beta_1H_x}{\varepsilon}  } \\
&\le
C\varepsilon^2 N^{1-2\rho},
\end{split}
\end{equation}
and 
\begin{equation}\label{eq:integral-II}
\begin{split}
&\int_{x_s}^{x_t}e^{-\frac{2\beta_1(1-x)}{\varepsilon}  }  \frac{x_t-x}{H_x}\mr{d}x
=\int_{x_s}^{x_t}e^{-\frac{2\beta_1(1-x_t)}{\varepsilon}  } e^{-\frac{2\beta_1(x_t-x)}{\varepsilon}  } \frac{x_t-x}{H_x} \mr{d}x\\
&\xlongequal{\xi=x_t-x}
N^{-2\rho}H_x^{-1} \int_0^{H_x}e^{-\frac{2\beta_1 \xi}{\varepsilon} }\xi \mr{d}\xi\\
&=N^{-2\rho} H_x^{-1}\left(\frac{\varepsilon}{2\beta_1}\right)^2
\left(
-   \frac{2\beta_1H_x}{\varepsilon} e^{-\frac{2\beta_1H_x}{\varepsilon}  } 
+ 1- e^{-\frac{2\beta_1H_x}{\varepsilon} }  \right)
\\
&\le
C\varepsilon^2 N^{1-2\rho}. 
\end{split}
\end{equation}
Similarly, we have
\begin{equation}\label{eq:integral-III}
\int_0^{y_s}e^{-\frac{2\beta_2(1-y)}{\varepsilon}  }\mr{d}y\le
C\varepsilon^2 N^{1-2\rho},
\quad
\int_{y_s}^{y_t}e^{-\frac{2\beta_2(1-y)}{\varepsilon}  }  \frac{y_t-y}{H_y}\mr{d}y
\le
C\varepsilon^2 N^{1-2\rho}. 
\end{equation}
According to the definition of $\delta$ in \eqref{eq: delta-K-SD-I}, we split 
$\mr{II}_4$ into two parts. Then 
\begin{equation}\label{eq:II4}
\begin{split}
|\mr{II}_4|
=&\big( \sum_{K\subset\Omega_{s,\varepsilon}}
+\sum_{K\subset\Omega^{c}_{s,\varepsilon}} \big)
\left|(\b{b} \cdot\nabla E,\delta\b{b} \cdot\nabla E)_{K}\right|\\
\le&
CN^{-1}\Vert \nabla E \Vert^2_{ \Omega_{s,\varepsilon} }
+C\iint \limits_{ \Omega^c_{s,\varepsilon} }\varepsilon^{-2}\left( e^{-\frac{2\beta_1(1-x)}{\varepsilon}  }+  e^{-\frac{2\beta_2(1-y)}{\varepsilon}  } \right)\delta(x,y)\mr{d}x\mr{d}y\\
\le & C N^{-2\rho},
\end{split}
\end{equation}
where we have used \eqref{eq:integral-I}--\eqref{eq:integral-III} and direct calculations.

Collecting \eqref{eq:II1}, \eqref{eq:II2}, \eqref{eq:II3} and \eqref{eq:II4}, we  obtain 
\begin{equation}\label{eq: stabilization in SD norm}
|\mr{II}|\le CN^{-3}.
\end{equation}

\noindent Combine \eqref{eq: u-uI varepsilon} and \eqref{eq: stabilization in SD norm}, then we are done.
\end{proof}

\begin{remark}
We can obtain local estimates in $\Omega_{s}$ 
 in the same way as in Lemma \ref{lem:u-uI in SD norm}. If Assumption \ref{assumption-regularity} holds true,  we have
\begin{equation}\label{eq: u-uI-varepsilon-Omega-s}
\Vert u-u^{I} \Vert^2_{\varepsilon;\,\Omega_{s}}
=\varepsilon \Vert \nabla (u-u^I) \Vert^2_{\Omega_s}+\mu_{0}\Vert  u-u^I  \Vert^2_{\Omega_s }
\le C(\varepsilon N^{-2}+N^{-4}).
\end{equation}
Note that analysis of $\varepsilon \Vert \nabla (u-u^I) \Vert^2_{\Omega_s}$ is similar to one of $\mr{II}$ in Lemma \ref{lem:u-uI in SD norm}.
Moreover, if $\delta$ is defined   in \eqref{eq: delta-K-SD-I}, the bounds 
\eqref{eq: u-uI-varepsilon-Omega-s} and \eqref{eq: stabilization in SD norm} yield
 \begin{equation}\label{eq: u-uI-SD-Omega-s}
 \begin{split}
\Vert u-u^{I} \Vert^2_{SD;\,\Omega_{s}}
&=\Vert u-u^{I} \Vert^2_{\varepsilon;\,\Omega_{s}}+\sum_{K\subset\Omega} (\b{b} \cdot\nabla (u-u^{I} ),\delta \b{b} \cdot\nabla (u-u^{I} ))_{K}\\
&\le CN^{-3}.
\end{split}
\end{equation}
\end{remark}

\begin{lemma}\label{lem:uI-uN in SD norm}
Let Assumption \ref{assumption-regularity} hold true and $\delta$ be defined in \eqref{eq: delta-K-SD-I} or in \eqref{eq: delta-K-usual-I},  then we have
$$
\Vert u^{I}-u^{N}\Vert_{SD} \le C(\varepsilon N^{-3/2}+N^{-2}\ln^{2}N).
$$
\end{lemma}
\begin{proof}
See the details in \cite[Theorem 4.5]{Styn1Tobi2:2003-SDFEM}.
\end{proof}

By the above lemmas we have the following theorem.
\begin{theorem}\label{theorem: u-uN-Omega s}
Let Assumption \ref{assumption-regularity} hold true. If $\delta$ is defined  in \eqref{eq: delta-K-SD-I},  we have
\begin{align}
&\Vert u-u^{N}\Vert_{SD} \le CN^{-1}\ln N,
\label{eq: u-uN SD}\\
&\Vert u-u^{N}\Vert_{SD;\,\Omega_{s}} \le CN^{-3/2}
\label{eq: u-uN SD Omega-s}.
\end{align}
If $\delta$ is defined  in \eqref{eq: delta-K-SD-I} or in \eqref{eq: delta-K-usual-I},  we have
\begin{equation}\label{eq: u-uN varepsilon-Omega-s}
\Vert u-u^{N}\Vert_{\varepsilon;\,\Omega_{s}} \le C (\varepsilon^{1/2}N^{-1}+N^{-2}\ln^{2}N).
\end{equation}
\end{theorem}
\begin{proof}
Combining Lemmas \ref{lem:u-uI in SD norm} and \ref{lem:uI-uN in SD norm},  we obtain
\eqref{eq: u-uN SD}.   Lemma \ref{lem:uI-uN in SD norm}  and  \eqref{eq: u-uI-SD-Omega-s}  yield 
\eqref{eq: u-uN SD Omega-s}. Note that $\Vert u^{I}-u^{N}\Vert_{\varepsilon}\le  \Vert u^{I}-u^{N} \Vert_{SD}$. Thus
from \eqref{eq: u-uI-varepsilon-Omega-s} and  Lemma \ref{lem:uI-uN in SD norm}, we have \eqref{eq: u-uN varepsilon-Omega-s}.
\end{proof}

%
%

%
%
\section{Numerical results}
\noindent
In this section we give the numerical results that appear to support our theoretical results. Errors and convergence rates in different norms are presented. Numerical experiments show that our new stabilization parameter  preserves  high accuracy and numerical stability as
the standard one. 

All calculations were carried out by using Intel visual Fortran 11. The discrete problems
were solved by the nonsymmetric iterative solver GMRES(c.f. e.g.,\cite{Ben1Gol2Lie3:2005-Numerical,Saad1Schu2:1986-GMRES}).

\noindent
\textbf{Problem}.
\begin{equation}\label{eq:problem-I}
\begin{split}
-\varepsilon\Delta u+2u_{x}+u_{y}+u=&f(x,y) \quad \text{in $\Omega=(0,1)^{2}$},\\
u=&0  \quad\quad\text{on $\partial\Omega$},
\end{split}
\end{equation}
where the right-hand side  $f$ is chosen such that
\begin{equation*}
u(x,y)=2\sin x\left(1-e^{-\frac{2(1-x)}{\varepsilon} } \right)
y^{2} \left( 1-e^{-\frac{(1-y)}{\varepsilon} } \right)
\end{equation*}
is the exact solution.

The errors in Tables \ref{table: vivi-10}--\ref{table: vivi-20} are measured as follows
\begin{equation*}
e^{N}_{SD}:= \left( \sum_{K\subset \Omega}\Vert u-u^{N} \Vert^{2}_{SD,K}\right)^{1/2},\;\;
e^{N}_{\varepsilon}:=\left( \sum_{K\subset \Omega}\Vert u-u^{N} \Vert^{2}_{\varepsilon,K}\right)^{1/2},
\end{equation*}
where $u^{N}$ is the SDFEM solution.

The corresponding rates of convergence $p^{N}$ are computed from the formula
\begin{equation}\label{eq: convergence order formula}
    p^{N}=\frac{\ln e^{N}-\ln e^{2N}}{\ln2},
\end{equation}
where $e^{N}$ can be $e^{N}_{SD}$, $e^{N}_{\varepsilon;\,\Omega_{s} }$ or $e^{N}_{SD;\,\Omega_{s} }$.

Tables \ref{table: vivi-10} and \ref{table: vivi-11} present the errors and convergence rates of $\Vert u-u^{N}\Vert_{\varepsilon;\,\Omega_{s} }$, which support the theorectical bound \eqref{eq: u-uN varepsilon-Omega-s}. Moreover, we  observe that if $\varepsilon\le N^{-2}$
 the convergence order of $\Vert u-u^{N}\Vert_{\varepsilon;\,\Omega_{s} }$ is almost $2$. If $\varepsilon\ge N^{-2}$,  maybe $\varepsilon^{1/2}N^{-1}$ dominates the bound of
 $\Vert u-u^N\Vert_{\varepsilon; \Omega_s}$.

Table \ref{table: vivi-12} gives the errors and convergence rates of  $\Vert u-u^{N} \Vert_{SD;\,\Omega_{s}}$, which show that
the convergence  order of $\Vert u-u^{N}\Vert_{SD;\,\Omega_{s} }$ is $3/2$.

\begin{table}
\caption{$\delta$ in \eqref{eq: delta-K-usual-I}}
\footnotesize
\begin{tabular*}{\textwidth}{@{\extracolsep{\fill}}c cc cc cc}
\hline
  & \multicolumn{2}{c}{$\varepsilon=10^{-4}$} &  \multicolumn{2}{c}{$\varepsilon=10^{-6}$} & \multicolumn{2}{c}{$\varepsilon=10^{-8},10^{-10},\ldots,10^{-16}$}\\
  \hline
$N$   & $ \Vert u-u^{N}\Vert_{\varepsilon;\,\Omega_{s}}$   & Rate   & $ \Vert u-u^{N}\Vert_{\varepsilon;\,\Omega_{s} }$    & Rate & $ \Vert u-u^{N}\Vert_{\varepsilon;\,\Omega_{s} }$    & Rate\\
\hline
8   &$7.79\times10^{-3}$  & $1.78$    &$7.65\times10^{-3}$  & $1.84$  &$7.64\times10^{-3}$  & $1.84$  \\
16  &$2.26\times10^{-3}$   &$1.87$  &$2.13\times10^{-3}$   &$2.14$ &$2.13\times10^{-3}$   &$2.14$\\
32  &$6.19\times10^{-4}$   &$1.46$  & $4.85\times10^{-4}$   &$2.05$  &$4.83\times10^{-4}$   &$2.06$\\
64  &$2.25\times10^{-4}$   &$1.16$ &$1.17\times10^{-4}$   &$1.97$   &$1.16\times10^{-4}$   &$2.03$\\
128  &$1.01\times10^{-4}$   &$1.04$  &$3.00\times10^{-5}$   &$1.81$ &$2.85\times 10^{-5}$   &$2.01$\\
256  &$4.87\times10^{-5}$   &$1.01$  &$8.55\times10^{-6}$   &$1.52$ &$7.07\times10^{-6}$   &$2.00$\\
512  &$2.41\times10^{-5}$   &---&$2.99\times10^{-6}$            &---&$1.77\times10^{-6}$   &---\\
\hline
\end{tabular*}
\label{table: vivi-10}
\end{table}

\begin{table}
\caption{$\delta$ in \eqref{eq: delta-K-SD-I}}
\footnotesize
\begin{tabular*}{\textwidth}{@{\extracolsep{\fill}}c cc cc cc}
\hline
  & \multicolumn{2}{c}{$\varepsilon=10^{-4}$} &  \multicolumn{2}{c}{$\varepsilon=10^{-6}$} & \multicolumn{2}{c}{$\varepsilon=10^{-8},10^{-10},\ldots,10^{-16}$}\\
  \hline
$N$   & $ \Vert u-u^{N}\Vert_{\varepsilon;\,\Omega_{s}}$   & Rate   & $ \Vert u-u^{N}\Vert_{\varepsilon;\,\Omega_{s} }$    & Rate & $ \Vert u-u^{N}\Vert_{\varepsilon;\,\Omega_{s} }$    & Rate\\
\hline
8   &$1.08\times10^{-2}$  & $1.82$    &$1.07\times10^{-2}$  & $1.85$  &$1.07\times10^{-2}$  & $1.85$  \\
16  &$3.05\times10^{-3}$   &$2.08$  &$2.96\times10^{-3}$   &$2.27$ &$2.96\times10^{-3}$   &$2.27$ \\
32  &$7.23\times10^{-4}$   &$1.63$  & $6.12\times10^{-4}$   &$2.19$  &$6.11\times10^{-4}$   &$2.20$\\
64  &$2.34\times10^{-4}$   &$1.21$ &$1.34\times10^{-4}$   &$2.06$   &$1.33\times10^{-4}$   &$2.12$\\
128  &$1.01\times10^{-4}$   &$1.05$  &$3.21\times10^{-5}$   &$1.87$ &$3.06\times10^{-5}$   &$2.06$\\
256  &$4.87\times10^{-5}$   &$1.01$  &$8.77\times10^{-6}$   &$1.55$ &$7.07\times10^{-6}$   &$2.00$\\
512  &$2.41\times10^{-5}$   &---&$3.01\times10^{-6}$            &---&$1.81\times10^{-6}$   &---\\
\hline
\end{tabular*}
\label{table: vivi-11}
\end{table}

\begin{table}
\caption{$\varepsilon=10^{-4},10^{-6},\ldots,10^{-16}$}
\footnotesize
\begin{tabular*}{\textwidth}{@{\extracolsep{\fill}}c cc cc}
\hline
  & \multicolumn{2}{c}{$\delta$ in \eqref{eq: delta-K-usual-I}  } &  \multicolumn{2}{c}{$\delta$ in \eqref{eq: delta-K-SD-I}  }\\
  \hline
$N$   & $ \Vert u-u^{N}\Vert_{SD;\,\Omega_{s}}$   & Rate   & $ \Vert u-u^{N}\Vert_{SD;\,\Omega_{s} }$    & Rate\\
\hline
8   &$1.80\times10^{-1}$  & $1.50$    &$1.37\times10^{-1}$  & $1.29$  \\
16  &$6.38\times10^{-2}$   &$1.50$  &$5.59\times10^{-2}$   &$1.41$   \\
32  &$2.25\times10^{-2}$   &$1.50$  & $2.11\times10^{-2}$   &$1.45$   \\
64  &$7.96\times10^{-3}$   &$1.50$ &$7.70\times10^{-3}$   &$1.48$  \\
128  &$2.81\times10^{-3}$   &$1.50$  &$2.77\times10^{-3}$   &$1.49$  \\
256  &$9.94\times10^{-4}$   & $1.50$  &$9.86\times10^{-4}$   &$1.49$  \\
512  &$3.52\times10^{-4}$   &---&$3.50\times10^{-4}$            &--- \\
\hline
\end{tabular*}
\label{table: vivi-12}
\end{table}

 In Table \ref{table: vivi-20}, the errors and convergence rates for $\Vert u-u^{N}\Vert_{SD}$ and $\Vert u-u^{N} \Vert_{\varepsilon}$ are displayed, which support \eqref{eq: u-uN SD}. We observe  similar bounds and convergence orders of $u-u^N$  with $\delta$ defined in \eqref{eq: delta-K-SD-I} or \eqref{eq: delta-K-usual-I}.

Plots \ref{fig:alex-1}---\ref{fig:alex-4} show that with the new  stabilization parameter $\delta$,  the SDFEM  solutions still preserve high accuracy and numerical stability.  Figures \ref{fig:alex-1} and \ref{fig:alex-2} present pointwise errors in the computational domain $\Omega$,  Figures \ref{fig:alex-3} and \ref{fig:alex-4}
in one exponential layer.
 These plots show that there are no visible oscillations in the SDFEM solutions with $\delta$ in \eqref{eq: delta-K-SD-I}, as the SDFEM solutions with $\delta$ in \eqref{eq: delta-K-usual-I}.

\begin{table}
\caption{$\varepsilon=10^{-4},10^{-6},\ldots,10^{-16}$}
\footnotesize
\begin{tabular*}{\textwidth}{@{\extracolsep{\fill}}c cc cc  cc  cc}
\hline
          &\multicolumn{4}{c}{$\delta$  in \eqref{eq: delta-K-usual-I} }  & \multicolumn{4}{c}{ $\delta$ in \eqref{eq: delta-K-SD-I}  }\\
          \hline
$N$    & $ \Vert u-u^{N}\Vert_{\varepsilon}$    & Rate & $ \Vert u-u^{N}\Vert_{SD}$   & Rate    & $ \Vert u-u^{N}\Vert_{\varepsilon}$    & Rate & $ \Vert u-u^{N}\Vert_{SD}$   & Rate  \\
\hline
8   &$4.19\times10^{-1}$  & $0.63$    &$4.56\times10^{-1}$      &    $0.71$
&$4.19\times10^{-1}$  & $0.63$    &$4.41\times10^{-1}$      &    $0.67$\\
16  &$2.71\times10^{-1}$   &$0.70$  &$2.79\times10^{-1}$      &    $0.73$
&$2.77\times10^{-1}$   &$0.70$  &$2.79\times10^{-1}$      &    $0.72$\\
32  &$1.67\times10^{-1}$   &$0.75$   &$1.68\times10^{-1}$     &    $0.76$
&$1.67\times10^{-1}$   &$0.75$   &$1.68\times10^{-1}$     &    $0.75$\\
64  &$9.93\times10^{-2}$   &$0.78$&$9.96\times10^{-2}$       &    $0.78$
&$9.93\times10^{-2}$   &$0.78$&$9.96\times10^{-2}$       &    $0.78$\\
128  &$5.78\times10^{-2}$   &$0.81$&$5.78\times10^{-2}$      &    $0.81$
&$5.78\times10^{-2}$   &$0.81$&$5.78\times10^{-2}$      &    $0.81$\\
256  &$3.30\times10^{-2}$   &0.83&$3.30\times10^{-2}$            &0.83
&$3.30\times10^{-2}$   &0.83&$3.30\times10^{-2}$            &0.83\\
512  &$1.85\times10^{-2}$   &---&$1.85\times10^{-2}$            &---
 &$1.85\times10^{-2}$   &---&$1.85\times10^{-2}$            &---\\
 \hline
\end{tabular*}
\label{table: vivi-20}
\end{table}

\begin{figure}
\begin{minipage}[t]{0.5\linewidth}
\centering
\includegraphics[width=2.5in]{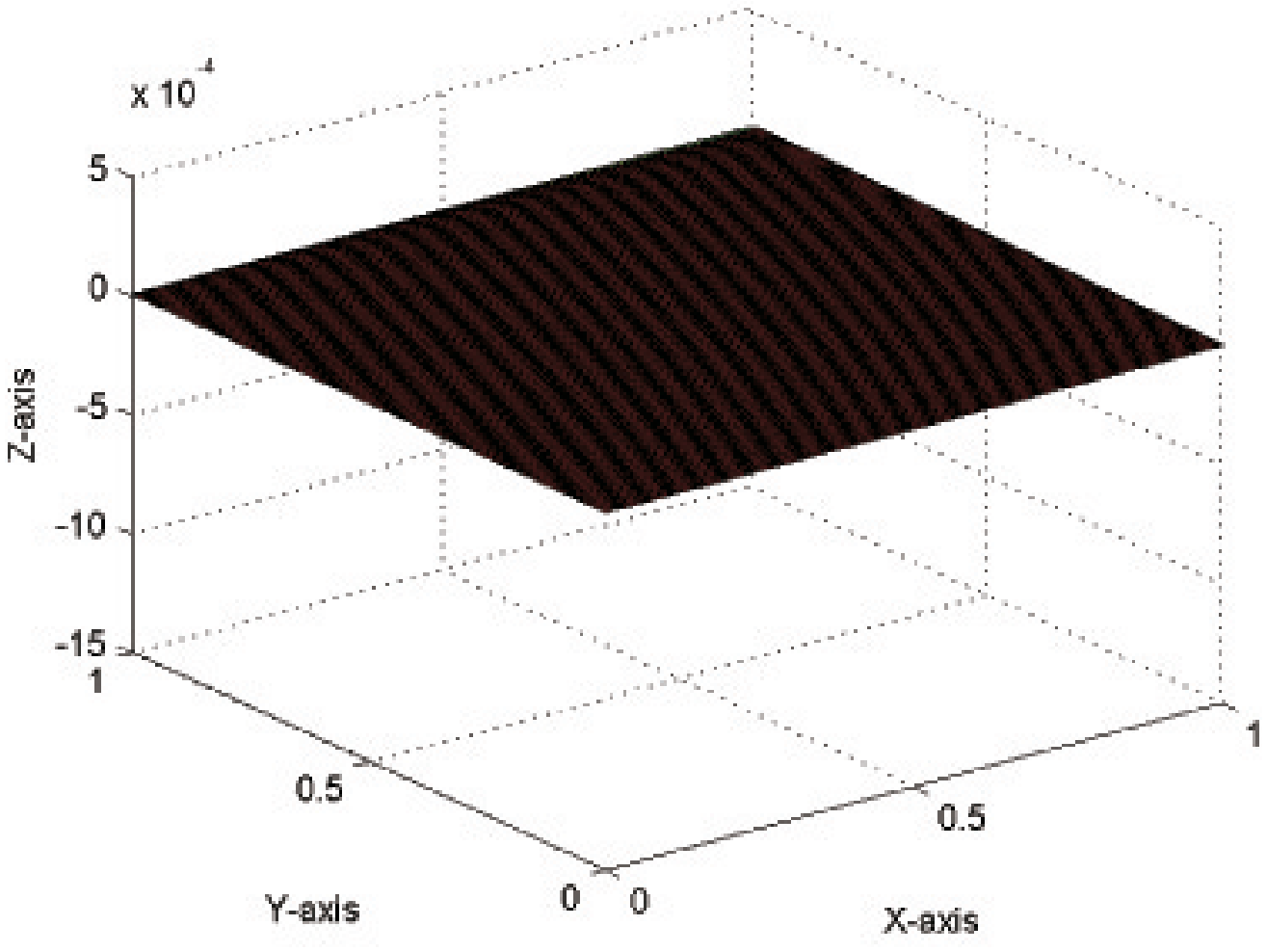}
\caption{Pointwise errors with $\delta$ in \eqref{eq: delta-K-usual-I} }
\label{fig:alex-1}
\end{minipage}%
\begin{minipage}[t]{0.5\linewidth}
\centering
\includegraphics[width=2.5in]{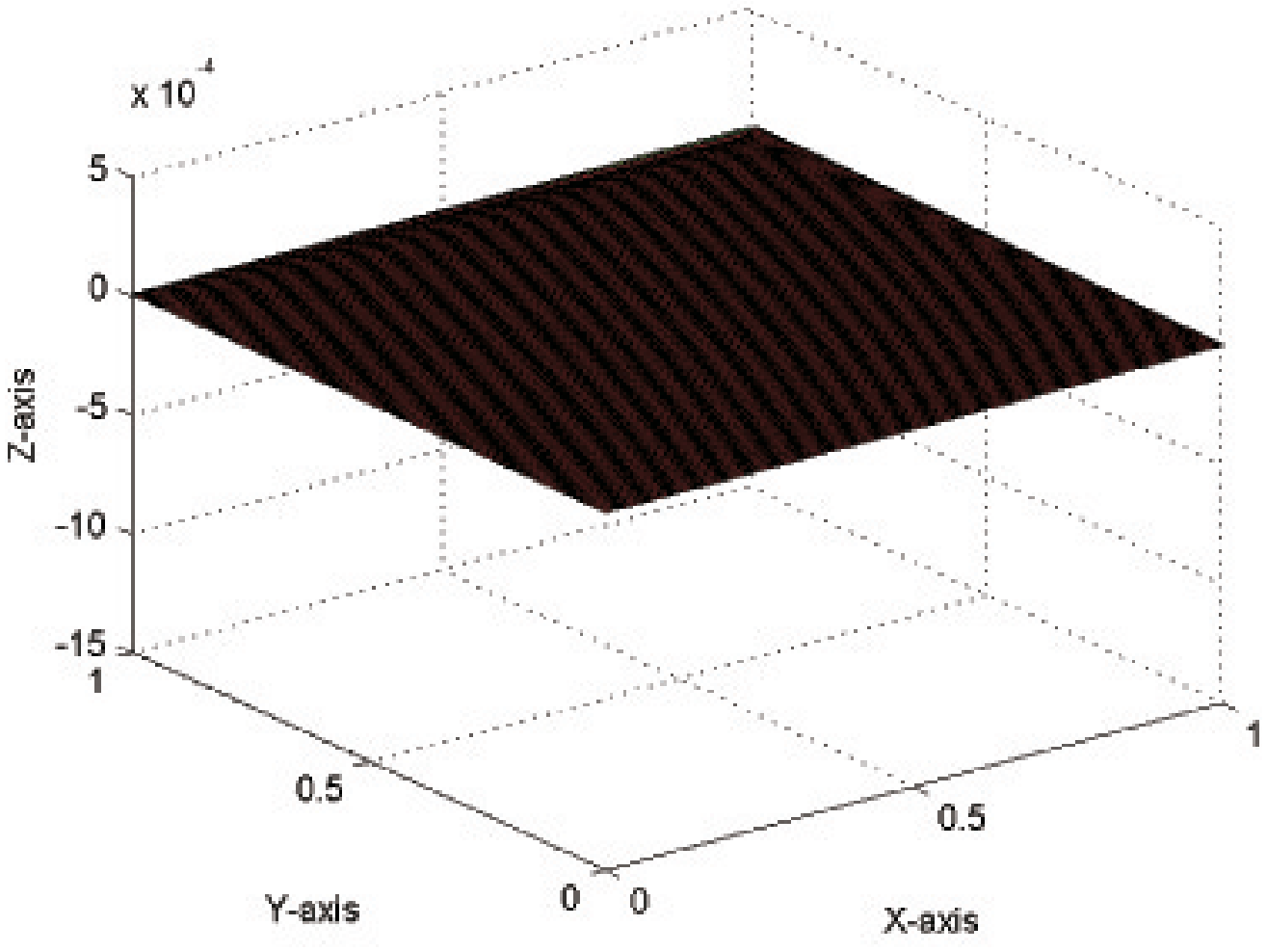}
\caption{Pointwise errors with $\delta$ in \eqref{eq: delta-K-SD-I} }
\label{fig:alex-2}
\end{minipage}
\end{figure}

\begin{figure}
\begin{minipage}[t]{0.5\linewidth}
\centering
\includegraphics[width=2.5in]{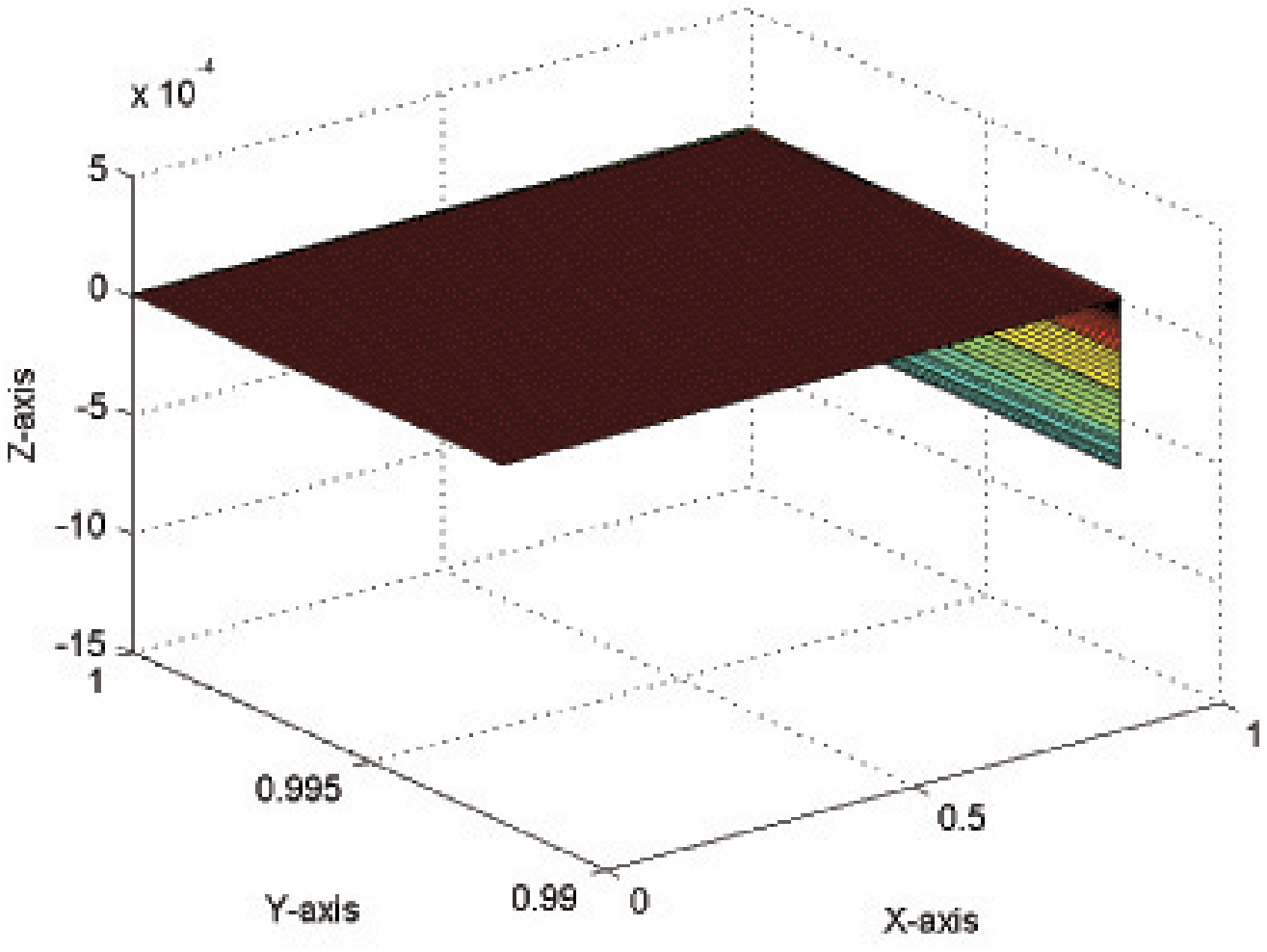}
\caption{Pointwise errors in exponential layers with $\delta$ in \eqref{eq: delta-K-usual-I} }
\label{fig:alex-3}
\end{minipage}%
\begin{minipage}[t]{0.5\linewidth}
\centering
\includegraphics[width=2.5in]{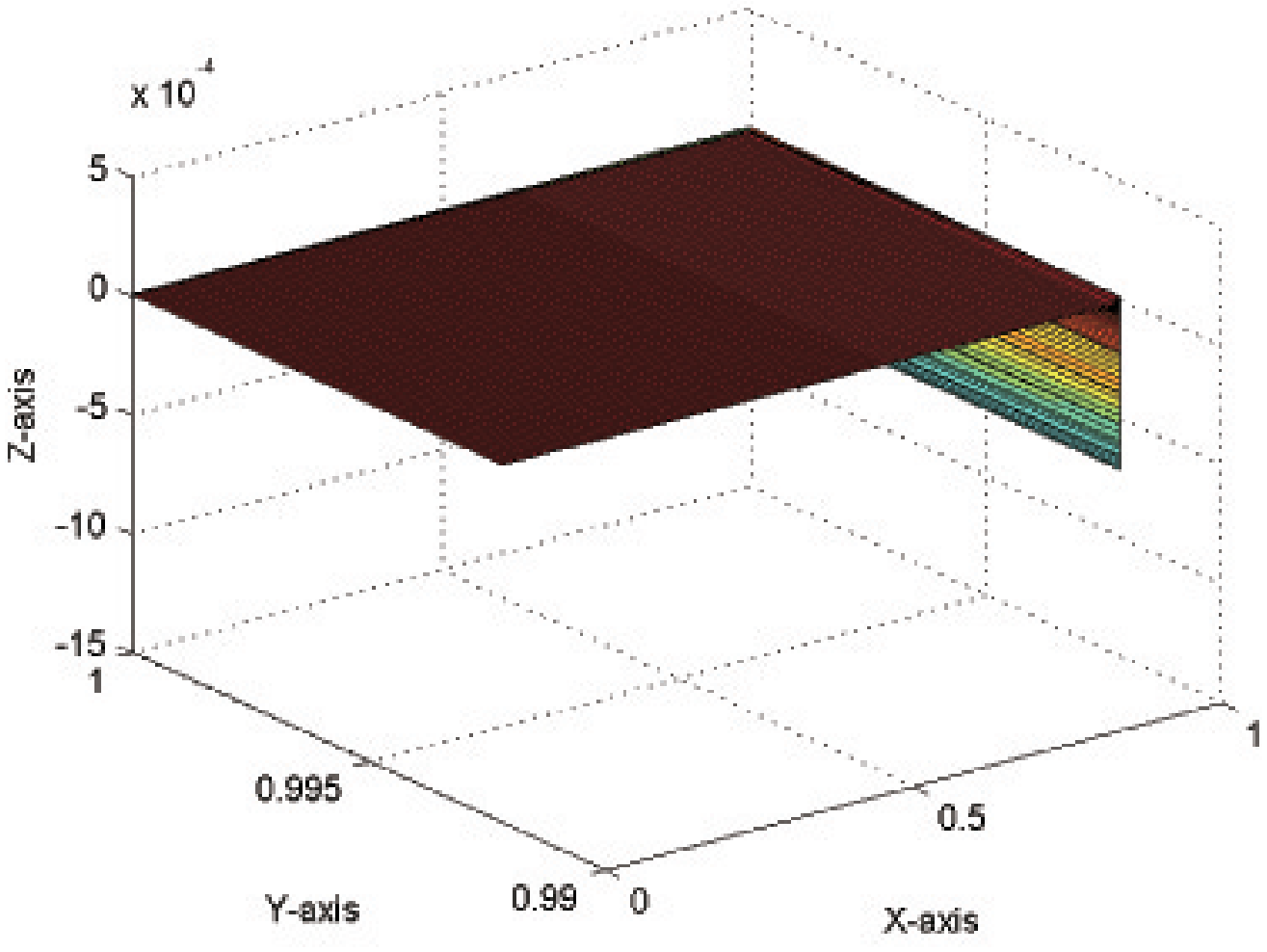}
\caption{Pointwise errors in exponential layers with $\delta$ in \eqref{eq: delta-K-SD-I}}
\label{fig:alex-4}
\end{minipage}
\end{figure}

%
%


%
%

\end{document}